\documentclass[12pt]{amsart}
\usepackage[english]{babel}
\usepackage{amsfonts,amssymb,latexsym,amscd}
\usepackage[all]{xy}
\usepackage{color}

\theoremstyle{plain}
\newtheorem{theorem}{Theorem}

\newtheorem{proposition}{Proposition}
\newtheorem{corollary}{Corollary}

\theoremstyle{definition}
\newtheorem{definition}{Definition}
\theoremstyle{remark}

\begin{document} 
 
\author{Pavel S. Gevorgyan}
\address{Moscow Pedagogical State University}
\email{pgev@yandex.ru}

\title [Bi-Equivariant Fibrations]{Bi-Equivariant Fibrations}

\begin{abstract}
The lifting problem for continuous bi-equivariant maps and bi-equivariant covering homotopies is considered, 
which leads to the notion of a bi-equivariant fibration. An intrinsic characteristic of a 
bi-equivariant Hurewicz fibration is obtained. Theorems concerning a relationship between bi-equivariant fibrations 
and fibrations generated by them are proved.
\end{abstract}


\maketitle

\section{Introduction}
						    
The foundations of bi-equivariant topology were laid in \cite{Gev}, where the notions  
of a binary $G$-space and of a bi-equivariant map were introduced. 
The space $H_2(X)$ of all continuous invertible binary operations on a locally compact locally 
connected space $X$ with left multiplication $AB(x,x')=A(x,B(x,x'))$ for $A,B\in H_2(X)$ and $x,x'\in X$ is 
a topological group which acts binarily on the space $X$. If a topological group $G$ acts  binarily and 
effectively on $X$, then $G$ is a subgroup of $H_2(X)$ \cite{Gev2}. 
All binary $G$-spaces and bi-equivariant maps form a category, 
which is a natural extension of the category of $G$-spaces and equivariant maps. Set-theoretic questions 
related to bi-equivariant topology were also considered in~\cite{Gev2}--\cite{Gev-Naz}.

This paper is devoted to the important problems of lifting continuous bi-equivariant maps and bi-equivariant 
covering homotopies, which lead to bi-equivariant fibrations. 

To study bi-equivariant Hurewicz fibrations, we introduce the  notion of a bi-equivariant covering function 
and prove that a map $p\colon E\to B$ is a bi-equivariant covering homotopy with respect 
to any binary $G$-space $X$ if and only if $p$ has a bi-equivariant covering function. 
This is an intrinsic characteristic of a bi-equivariant Hurewicz fibration, because the definition 
of a bi-equivariant covering function for a map $p\colon E\to B$ is not related to a binary 
$G$-space~$X$. 

The property of being a bi-equivariant covering homotopy is preserved under the passage to 
closed subgroups of $G$; i.e., a bi-equivariant 
Hurewicz $G$-fibration $p\colon E\to B$ is also a bi-equivariant~$H$-fibration.

In the case where $E$ and $B$ are distributive binary $G$-spaces, a bi-equivariant surjective 
map $p\colon E\to B$ induces a bi-equivariant map $p^H\colon E^H\to B^H$ of the spaces of $H$-fixed points, 
which is a bi-equivariant Hurewicz $G$-fibration if so is $p\colon E\to B$. 

In the case where $G$ is a compact Abelian Lie group and $E$ and $B$ are distributive binary $G$-spaces, 
a bi-equivariant 
Hurewicz fibration $p\colon E\to B$ induces a Hurewicz fibration $p^* \colon E^* \to B^*$ of the orbit spaces.


\section{Auxiliary Definitions and Results}

\subsection{Binary $G$-Spaces and bi-equivariant Maps} 
Let $G$ be any topological group, and let $X$ be any topological space.

A continuous map $\alpha \colon G\times X^2\to X$ is called a \textit{binary action} 
of the topological group $G$ on the space $X$ if, for any $g,h\in G$ and $x_1,x_2\in X$, 
\begin{equation}\label{eq(1)}
gh(x,x')=g(x, h(x,x')),
\end{equation}
\begin{equation}\label{eq(2)}
e(x,x')=x',
\end{equation}
where $e$ is the identity element of $G$ and $g(x,x') = \alpha (g, x,x')$.

A triple $(G,X,\alpha)$ consisting of a space $X$,  a group $G$, and a fixed binary action $\alpha$ of $G$ on $X$ 
is called a \textit{binary $G$-space}.

Let $(G,X,\alpha)$ and $(G,Y,\beta)$ be binary $G$-spaces.
A continuous map $f\colon X\to Y$ is said to be \textit{bi-equivariant}, or \emph{$G$-bi-equivariant}, 
if 
$$f(\alpha(g,x,x'))=\beta(g,f(x),f(x')),$$
or
$$f(g(x,x'))=g(f(x),f(x')),$$
for all $g\in G$ and $x,x' \in X$.

All binary $G$-spaces and bi-equivariant maps form a category, which we denote by Bi-$G$-$Top$.

Note that, given any $G$-space $(X,G,\alpha)$, the unary action $\alpha$ on $X$ 
generates the binary action $\overline{\alpha}$ defined by
\begin{equation}\label{eq11}
\overline{\alpha}(g,x,x')=\alpha(g, x') \quad \text{or} \quad g(x,x')=gx'
\end{equation}
for all $g\in G$ and $x, x'\in X$, 
and any equivariant map $f:(G,X,\alpha)\to (G,Y,\beta)$ between $G$-spaces $X$ and $Y$
is bi-equivariant with respect to the action~\eqref{eq11}. Indeed,
$$f(\overline{\alpha}(g,x,x'))=f(\alpha(g, x') )=\beta(g,f(x'))=\overline{\beta} (g,f(x), f(x')).$$

Thus, the category Bi-$G$-$Top$ can be considered  as a natural extension of the category 
$G$-$Top$ of all $G$-spaces and equivariant maps.

Let $H$ be a subgroup of a group $G$. Then any binary $G$-space is also a binary $H$-space, 
and any $G$-bi-equivariant map is an $H$-bi-equivariant map. Thus, there exists a natural covariant 
functor from the category Bi-$G$-TOP to the category Bi-$H$-TOP.

Let $X$ be a binary $G$-space, and let $A$ be its subset. The set 
$G(A,A) = \{g(a,a'); g\in G, a,a'\in A\}$ is called the \emph{binary saturation} of $A$. 
A subset $A$ of a binary $G$-space $X$ is said to be \emph{bi-invariant}, or \emph{$G$-bi-invariant}, 
if $A$ coincides with its binary saturation: $G(A,A)=A$. A bi-invariant subset $A\subset X$ 
is itself a binary $G$-space;  it is called a \emph{binary $G$-subspace} of the space~$X$.

A binary $G$-space $X$ is said to be \textit{distributive} if, for any $x, x', x'' \in X$ and $g, h\in G$, 
we have 
\begin{equation}\label{eq1-1}
g(h(x,x'), h(x,x''))=h(x,g(x', x'')).
\end{equation}

The class of distributive binary $G$-spaces plays an important role in the theory of binary $G$-spaces. 
This is explained, in particular, by the special role of distributive subgroups of the group $H_2(X)$ 
of all continuous invertible binary operations on $X$. For example, any topological group 
is a distributive subgroup of the group of invertible binary operations on some space \cite{Gev1}. 
This statement is a binary topological analogue of Cayley's classical theorem on the representation 
of any finite group by unary operations (permutations).

\subsection{The space of $H$-fixed points of a binary $G$-space}

Let $X$ be a binary $G$-space, and let $H$ be  a subgroup of $G$. The set 
\[
X^H=\{x'\in X; \quad h(x,x')=x' \ \text{for all} \ h\in H,\ x\in X\}
\]
is called the \emph{space of $H$-fixed points} of the binary $G$-space $X$. This set 
is not generally $G$-bi-invariant. However, the  following assertion is valid.

\begin{proposition}\label{prop-1}
If $X$ is a distributive binary $G$-space, then $X^H$ is a $G$-bi-invariant subset and, therefore, 
a binary $G$-subspace.
\end{proposition}

\begin{proof}
Let $x',x''\in X^H$ be any $H$-fixed points. Then, since the binary action on $X$ is distributive, 
it follows that, for any $g\in G$, $h\in H$ and $x\in X$, we have 
\[
h(x, g(x',x'')) = g (h(x,x'), h(x,x''))= g(x',x''),
\]
i.e., $ g(x',x'') \in X^H$. Thus, $G(X^H, X^H)= X^H$.
\end{proof}

\begin{proposition}\label{prop-2}
If $X$ and $Y$ are binary $G$-spaces and $f\colon X\to Y$ is a surjective bi-equivariant map, 
then $f(X^H)\subset Y^H$. 
\end{proposition}

\begin{proof} 
Indeed, for any $x'\in X^H$, we have 
$$h(y, f(x')) = h(f(x), f(x')) = f(h(x, x')) = f(x'),$$
where $y\in Y$ is any element and $x\in X$ is a preimage of $y$. Therefore, $f(x')\in Y^H$. 
\end{proof}

We denote the restriction of a surjective bi-equivariant map $f\colon X\to Y$ to $X^H$ by $f^H\colon X^H\to Y^H$. 
Thus, if $X$ is a distributive binary $G$-space, then $f^H\colon X^H\to Y^H$ is a bi-equivariant map 
of binary $G$-spaces.

\subsection{The orbit space of a distributive binary $G$-space}

Let $X$ be a binary $G$-space. The orbit of a point $x\in X$ is the minimal bi-invariant subset $[x]$ of 
$X$ containing the point $x$. Obviously, $x\in G(x,x)\subset [x]$ for all $x\in X$. Therefore, 
if $G(x,x)$ is a bi-invariant set, then $G(x,x)= [x]$. 

As is known, the set $G(x,x)$ is not generally bi-invariant \cite[Example 2]{Gev-Naz}. 
However, in distributive binary $G$-spaces the sets $G(x,x)$ are bi-invariant, and hence 
$[x]=G(x,x)$ \cite[Theorem 9]{Gev2}. Moreover, the orbits of a distributive binary $G$-space 
either are disjoint or coincide \cite[Proposition 6]{Gev-Naz}; therefore, the space $X$ is 
partitioned into disjoint classes. 
We denote the corresponding quotient set by~$X|G$. 

Let $\pi=\pi_X:X\to X|G$ be the natural quotient map which sends each point $x\in X$ to its orbit $[x]$. 
The quotient topology on $X|G$ is defined in the standard way. The space thus obtained is called 
the \emph{orbit space} of the distributive binary $G$-space~$X$.

If $X$ and $Y$ are distributive binary $G$-spaces, then any bi-equivariant map   
$f:X\to Y$ generates a map $f^*:X|G \to Y|G$ of the orbit spaces, which is defined by 
$f^*([x])=[f(x)]$. This map is well defined, because it does not depend on the choice 
of representatives of orbits: $f^*([g(x,x)]) = [f(g(x,x))]=[g(f(x),f(x))]=[f(x)]=f^*([x])$ for any $g\in G$.

All notions, definitions, and results used in the paper without references, as well as all those mentioned above, 
can be found in \cite{Gev2}--\cite{Gev-Naz}. 

\subsection{The bi-equivariant covering homotopy property and bi-equivariant Hurewicz fibrations}

Important problems of algebraic topology are the problem of \emph{lifting} continuous maps and 
the \emph{covering homotopy problem}, which is dual to the problem of extending homotopies. 
The covering homotopy problem leads to the notion of a fibration. We study these problems 
in the category Bi-$G$-TOP of binary $G$-spaces and bi-equivariant maps.

Let $G$ be a topological group. Suppose given  binary $G$-spaces $E$, $B$ and $X$ and bi-equivariant maps 
$p\colon E\to B$ and $f\colon X\to B$. The \emph{problem of lifting} the bi-equivariant map $f$ to $E$ 
consists in determining whether there exists a continuous bi-equivariant map $\tilde{f}:X\to E$ such that 
$f=p\circ \tilde{f}$; this map is denoted by dashed arrow in the diagram
$$
\xymatrix
{
& & E \ar[d]^{p} \\
X \ar[rr]_{f} \ar@{-->}[rru]^{\tilde{f}} & & B .
}
$$

To turn this problem into a well-posed problem of binary $G$-homotopy category, we need the counterpart
of the \emph{homotopy extension property}, which is called the \emph{bi-equivariant 
covering homotopy property}. We say that a bi-equivariant map $p : E \to B$  has 
the \emph{bi-equivariant covering homotopy property} with respect to a binary $G$-space $X$ if, 
for any bi-equivariant maps $\tilde{f}:X \to E$ and $F: X\times I \to B$ such 
that $F\circ i_0=p\circ \tilde{f}$, where $i_0:X\to X\times I$ is the embedding defined by 
$i_0(x)=(x,0)$, $x\in X$, there exists a bi-equivariant map $\tilde{F} : X\times I \to  E$  for which 
the diagram
$$
\xymatrix
{
X \ar[rr]^{\tilde{f}} \ar[d]_{i_0} & & E \ar[d]^{p} \\
X\times I \ar[rr]_{F} \ar@{-->}[rru]^{\tilde{F}} & & B 
}
$$
is commutative, i.e.,  $p\circ \tilde{F} =F$ and $\tilde{F}\circ i_0=\tilde{f}$. 

Obviously, if $p : E \to B$ has the bi-equivariant covering homotopy property with respect to 
a binary $G$-space $X$ and $f,g:X\to B$ are $G$-homotopic bi-equivariant maps, then $f$ lifts to 
$E$ if and only if so does $g$. Thus, the existence of a lifting of a bi-equivariant map $f:X\to B$ is 
a property of a binary $G$-homotopy class of this map.

The bi-equivariant covering homotopy property  leads to the notion of a bi-equivariant fibration. 
A bi-equivariant map $p : E \to B$ is called a \emph{bi-equivariant Hurewicz fibration}, or 
a \emph{bi-equivariant Hurewicz $G$-fibration}, if $p$ has the bi-equivariant covering homotopy property 
with respect to any binary $G$-space $X$.  In this case, the binary $G$-space $E$ is 
called the \emph{space of a bi-equivariant fibration}, and $B$ is the \emph{base} of this fibration.


\section{Covering Binary $G$-Functions\\ and an Intrinsic Characteristic of bi-equivariant Fibrations}

Let $G$ be a compact topological group. Suppose given  binary $G$-spaces  $E$ and $B$ and 
a bi-equivariant map $p : E \to B$. 

For the path space $B^I$, we define a map $G\times B^I\times B^I \to B^I$ by 
\begin{equation}\label{eq-prputey}
g(\alpha, \alpha')(t)=g(\alpha(t), \alpha'(t)),
\end{equation}
where $g\in G$, $\alpha, \alpha' \in B^I$, and $t\in I$. 

\begin{proposition}
The map \eqref{eq-prputey} defines a continuous binary action of the group $G$ on the path space~$B^I$.
\end{proposition}

\begin{proof}
The continuity of the map \eqref{eq-prputey} follows from that of $\alpha$ and $\alpha'$ and of the 
binary action of $G$ on~$B$.

Let us check that the action \eqref{eq-prputey} is binary. For any $g,h\in G$, $\alpha,\alpha'\in B^I$, 
and $t,t'\in I$, we have

(1) $e(\alpha,\alpha')(t)=g(\alpha(t), \alpha'(t))=\alpha'(t)$, i.e., $e(\alpha,\alpha') = \alpha'$; 
\begin{multline*}
(2) \ gh(\alpha,\alpha')(t) = gh(\alpha(t),\alpha'(t))=g(\alpha(t), h(\alpha(t),\alpha'(t)))= \\
= g(\alpha(t), h(\alpha,\alpha')(t))= g(\alpha, h(\alpha,\alpha'))(t), 
\end{multline*}
i.e., $gh(\alpha,\alpha') = g(\alpha, h(\alpha,\alpha'))$.
\end{proof}

Since $E$ and $B^I$ are binary $G$-spaces, it follows that the product $E\times B^I$ 
is a binary $G$-space with coordinatewise binary action. Consider the subspace $\Delta\subset E\times B^I$ 
defined by 
\[
\Delta=\{(e,\alpha)\in E\times B^I; \ \alpha(0) = p(e)\}.
\]
\begin{proposition}
The set $\Delta$ is a bi-invariant subspace of the binary $G$-space $E\times B^I$. 
\end{proposition}

\begin{proof}
Let $(e,\alpha), (e',\alpha') \in \Delta$ be any elements, i.e., $\alpha(0) = p(e)$ 
and $\alpha'(0) = p(e')$. Then $g((e,\alpha), (e',\alpha'))= (g(e,e'), g(\alpha,\alpha'))$ and 
\[
g(\alpha,\alpha')(0)=g(\alpha(0), \alpha'(0))=g(p(e), p(e'))=p(g(e,e')).
\]
Therefore, $g((e,\alpha), (e',\alpha'))\in \Delta$ for any~$g\in G$.
\end{proof}

\begin{definition}\label{def-G-function}
A bi-equivariant map $\lambda : \Delta \to E^I$  satisfying the conditions 
\begin{equation}\label{eq-lambda}
\lambda (e,\alpha)(0) = e \quad  \text{and} \quad  [p\circ \lambda (e,\alpha)](t)=\alpha(t) 
\end{equation}
is called a \emph{bi-equivariant covering function}, or a \emph{covering bi-equivariant $G$-function}, for~$p$.
\end{definition} 

The following theorem describes a relationship between bi-equivariant Hurewicz fibrations and 
bi-equivariant covering functions.

\begin{theorem}\label{th-1}
A bi-equivariant map $p : E \to B$ is a bi-equivariant Hurewicz fibration if and only if $p$ 
has a bi-equivariant covering function.
\end{theorem}

\begin{proof}
Let $\lambda : \Delta \to E^I$ be a bi-equivariant covering function for $p$. 
Consider any binary $G$-space $X$, bi-equivariant map $\tilde{f}:X\times 0 \to E$, and bi-equivariant 
homotopy $F: X\times I \to B$ for which $F(x,0)=p(\tilde{f}(x,0))$. 

Note that, for any $x\in X$, the formula $F_x(t)=F(x,t)$ defines a path $F_x\in B^I$. 
It follows from the bi-equivariance of $F$ that 
\begin{equation}\label{eq-F_g}
F_{g(x,x')} = g(F_x, F_{x'}).
\end{equation}
Indeed, 
\[
F_{g(x,x')} (t) = F(g(x,x'),t)= g(F(x,t), F(x',t)) = g(F_x(t), F_{x'}(t)) = g(F_x, F_{x'})(t).
\]

Now consider the homotopy $\tilde{F} : X\times I \to  E$ defined by 
\begin{equation}\label{eq-F(x,t)}
\tilde{F}(x,t) = \lambda(\tilde{f}(x,0), F_x)(t).
\end{equation}
Let us prove that $\tilde{F}$ is the required bi-equivariant covering homotopy.

The bi-equivariance of the homotopy $\tilde{F}$ follows from that of the maps $\lambda$, $\tilde{f}$, and $F$ 
and relations \eqref{eq-prputey}, \eqref{eq-F_g}, and \eqref{eq-F(x,t)}:
\begin{multline*}
\tilde{F}(g(x,x'),t) = \lambda(\tilde{f}(g(x,x'),0), F_{g(x,x')})(t) =  \lambda(g(\tilde{f}(x,0),\tilde{f}(x',0)), g(F_x,F_x'))(t) = \\
= \lambda (g( (\tilde{f}(x,0),F_x), (\tilde{f}(x',0), F_x') )) (t) = g( \lambda (\tilde{f}(x,0),F_x) , \lambda ( \tilde{f}(x',0), F_x' ) )(t) = \\
= g( \lambda (\tilde{f}(x,0),F_x)(t) , \lambda ( \tilde{f}(x',0), F_x' )(t) ) = g(\tilde{F}(x,t), \tilde{F}(x',t)).
\end{multline*}

According to conditions \eqref{eq-lambda}, we have
\[
\tilde{F}(x,0) = \lambda(\tilde{f}(x,0), F_x)(0) = \tilde{f}(x,0),
\]
\[
(p\circ \tilde{F})(x,t)= p(\tilde{F}(x,t)) = p(\lambda(\tilde{f}(x,0), F_x)(t)) = [p\circ \lambda(\tilde{f}(x,0), F_x)](t) = F_x(t) = F(x,t).
\]
Therefore, $\tilde{F}$ is a covering homotopy.

Now suppose that $p : E \to B$ is a bi-equivariant Hurewicz fibration. Consider the binary $G$-space $X=\Delta$ 
and the maps
\[
\tilde{f} : \Delta \times 0 \to E  \quad \text{and}  \quad F:\Delta \times I \to B 
\]
defined by 
\begin{equation}\label{eq-fF}
\tilde{f} [(e,\alpha),0] = e \quad \text{and}  \quad  F[(e,\alpha),t]=\alpha(t).
\end{equation}
The maps $\tilde{f}$ and $F$ are bi-equivariant. Indeed, for any $g\in G$, 
$(e,\alpha), (e',\alpha') \in \Delta \subset E\times B^I$, and $t\in I$, we have
\[
\tilde{f} [g((e,\alpha), (e',\alpha')), 0] = \tilde{f} [(g(e,e'), g(\alpha,\alpha')), 0] = g(e,e') = g(\tilde{f} [(e,\alpha),0], \tilde{f} [(e',\alpha'),0]),
\]
\begin{multline*}
F [g((e,\alpha), (e',\alpha')), t] = F [(g(e,e'), g(\alpha,\alpha')), t] = g(\alpha,\alpha')(t) = \\
= g(\alpha(t),\alpha'(t))  = g(F[(e,\alpha),t], F[(e',\alpha'),t]).
\end{multline*}

Note that $F:\Delta \times I \to B$ is a bi-equivariant homotopy of the map 
$p\circ \tilde{f} : \Delta \times 0 \to B$. Indeed,
\[
F[(e,\alpha),0]=\alpha(0)=p(e)= p(\tilde{f} [(e,\alpha),0]) = (p\circ \tilde{f}) [(e,\alpha),0].
\]
Hence there exists a bi-equivariant covering homotopy  \mbox{$\tilde{F} : \Delta \times I \to  E$} 
for $F$, i.e.,  
\begin{equation}\label{eq-1}
\tilde{F}[(e,\alpha),0]=\tilde{f}[(e,\alpha),0] \quad \text{and} \quad p \circ \tilde{F} = F.
\end{equation}
Now consider the map  $\lambda : \Delta \to E^I$ defined by 
\begin{equation}\label{eq-2}
\lambda (e,\alpha)(t)=\tilde{F}[(e,\alpha),t].
\end{equation}
Let us prove that $\lambda$ is a covering bi-equivariant $G$-function for $p$. The bi-equivariance of $\lambda$ 
follows from that of the covering homotopy $\tilde{F}$:
\begin{multline*}
\lambda (g((e,\alpha), (e',\alpha')))(t)=\tilde{F}[g((e,\alpha), (e',\alpha')),t] = \\
= g(\tilde{F}[(e,\alpha),t], \tilde{F}[(e',\alpha'),t]) = g(\lambda (e,\alpha)(t), \lambda (e',\alpha')(t)).
\end{multline*}
Conditions \eqref{eq-lambda} in Definition~\ref{def-G-function} also hold. By virtue of 
\eqref{eq-fF}, \eqref{eq-1}, and \eqref{eq-2}, we have
\[
\lambda (e,\alpha)(0) = \tilde{F}[(e,\alpha),0] = \tilde{f}[(e,\alpha),0] = e,
\]
\[
[p\circ \lambda (e,\alpha)](t)= p( \lambda (e,\alpha)(t)) = p(\tilde{F}[(e,\alpha),t]) = F[(e,\alpha),t] =  \alpha(t).
\]

This completes the proof of the theorem.
\end{proof}

Note that the last theorem gives an intrinsic characteristic of a bi-equivariant Hurewicz fibration, 
because Definition~\ref{def-G-function} of a bi-equivariant covering function does not involve any outer 
space~$X$.  

Applying Theorem~\ref{th-1} to the trivial binary action, we  obtain the following result.

\begin{corollary}\label{cor-1}
A continuous map $p : E \to B$ is a Hurewicz fibration if and only if $p$ has a covering function.
\end{corollary}


\section{Fibrations Generated by bi-equivariant $G$-Fibrations}

Let $H$ be a closed subgroup of a compact group $G$. Since all binary $G$-spaces are also binary $H$-spaces, \
and all $G$-bi-equivariant maps are $H$-bi-equivariant maps, there arises the natural question of whether 
the property of being a bi-equivariant Hurewicz foliation is preserved under the passage to a closed subgroup~$H$. 

\begin{theorem}\label{th-5-0}
Let $p : E \to B$ be a bi-equivariant Hurewicz $G$-fibration. Then, for any closed subgroup $H$ 
of the compact group $G$, the map $p : E \to B$ is also a bi-equivariant Hurewicz $H$-fibration.
\end{theorem}

\begin{proof}
By virtue of Theorem~\ref{th-1}, there exists a covering bi-equivariant $G$-function 
$\lambda : \Delta \to E^I$ for $p$. Since $\lambda : \Delta \to E^I$ is also a bi-equivariant $H$-map, 
it follows that $\lambda$ is a covering bi-equivariant $H$-function for the bi-equivariant $H$-map $p$. 
Therefore, $p : E \to B$  is a bi-equivariant Hurewicz $H$-fibration by virtue of the same Theorem~\ref{th-1}.
\end{proof}

Under certain constraints, the property of being a bi-equivariant Hurewicz $G$-fibration  
is preserved under the passage to binary $G$-subspaces of $H$-fixed points.

\begin{theorem}\label{th-5}
Let $E$ and $B$ be a distributive binary $G$-space, and let $p : E \to B$ be a surjective bi-equivariant 
Hurewicz $G$-fibration. Then, for any closed subgroup $H$ of the compact group $G$, the induced 
bi-equivariant $G$-map $p^H : E^H \to B^H$ between the spaces of $H$-fixed points is a bi-equivariant 
 Hurewicz $G$-fibration.
\end{theorem}

\begin{proof}
The bi-invariance of the set of $H$-fixed points of a distributive binary $G$-space and the preservation of 
$H$-fixed points under a surjective bi-equivariant map were proved in Propositions~\ref{prop-1} and~\ref{prop-2}.

Let $\lambda : \Delta \to E^I$ be a covering bi-equivariant $G$-function for $p$, whose existence follows 
from Theorem~\ref{th-1}.  

Consider the set 
$$\Delta^H=\{(\dot{e},\dot{\alpha})\in E^H\times (B^H)^I; \ \dot{\alpha}(0) = p^H(\dot{e})\}.$$

Let us prove the existence of a covering bi-equivariant $G$-function 
$\lambda^H : \Delta^H \to (E^H)^I$ for a bi-equivariant map $p^H : E^H \to B^H$.

Since $\Delta^H\subset \Delta$, we can set $\lambda^H = \lambda|\Delta^H$. It is easy to see 
that $\lambda^H$ takes the set  $\Delta^H$ to $(E^H)^I$ and that 
$\lambda^H : \Delta^H \to (E^H)^I$ is a covering bi-equivariant $G$-function for the map  
$p^H : E^H \to B^H$. Therefore, by virtue of Theorem \ref{th-1}, $p^H : E^H \to B^H$ is a bi-equivariant 
Hurewicz $G$-fibration.
\end{proof}

In the case of distributive binary $G$-spaces, the property of being a binary Hurewicz $G$-fibration 
is also preserved under the passage to orbit spaces.

\begin{theorem}\label{th-5-1}
Let $G$ be a compact Abelian Lie group, and let $E$ and $B$ be distributive binary $G$-spaces. 
If $p : E \to B$ is a bi-equivariant Hurewicz $G$-fibration, then the induced map  
$p^* : E|G \to B|G$ of orbit spaces is a Hurewicz fibration.
\end{theorem}

\begin{proof}
Let  $p : E \to B$ be a bi-equivariant Hurewicz $G$-fibration. By Theorem~\ref{th-1} there exists a 
covering bi-equivariant $G$-function $\lambda : \Delta \to E^I$ for $p$. 

Consider the set 
$$\Delta_G=\{(e^*,\alpha^*)\in E|G\times (B|G)^I; \ \alpha^*(0) = p^*(e^*)\}.$$
Let us prove that there exists a covering function  $\lambda^*:\Delta_G \to (E|G)^I$ for the 
map $p^* : E|G \to B|G$.

Since $B$ is a distributive binary $G$-space, it follows that its orbits have the form $G(x,x)$ for any 
$x\in B$. Note that $B|G$ can also be treated as the orbit space of the $G$-space $B$ with the action 
$gx = g(x,x)$, $g\in G$, $x\in B$. This formula indeed defines an action of the group $G$ on $B$,  because 
the group $G$ is commutative and the  binary action of the group $G$ on $B$ is distributive. 

Note that any path $\alpha^*:I \to B|G$ can be lifted to $B$ (\cite[Theorem 6.2]{Bredon}), i.e., 
there exists a path $\alpha : I\to B$ such  that $\pi_B \circ \alpha = \alpha^*$, where 
$\pi_B: B\to B|G$ is the orbit projection. Now we define the required map $\lambda^*:\Delta_G \to (E|G)^I$ by
$$\lambda^*(e^*,\alpha^*)(t)=(\lambda(e,\alpha)(t))^*.$$
The map $\lambda^*$ is well defined, and it is a covering function for $p^* : E|G \to B|G$. 
Therefore, by Corollary~\ref{cor-1}, $p^* : E|G \to B|G$ is a Hurewicz fibration.
\end{proof}


\end{document}